\documentclass{amsart}

\usepackage{amsmath,amssymb,url,color,mathtools}
\usepackage{tikz}

\newtheorem{lemma}{Lemma}[section]
\newtheorem{prop}[lemma]{Proposition}
\newtheorem{cor}[lemma]{Corollary}
\newtheorem{thm}[lemma]{Theorem}
\newtheorem{example}[lemma]{Example}
\newtheorem{thm?}[lemma]{Theorem?}

\newtheorem{ques}[lemma]{Question}

\newtheorem{remark}[lemma]{Remark}

\begin{document}
\title{Restricted Variable Chevalley-Warning Theorems}

\author{Anurag Bishnoi}
\author{Pete L. Clark}

\newcommand{\etalchar}[1]{$^{#1}$}
\newcommand{\F}{\mathbb{F}}
\newcommand{\et}{\textrm{\'et}}
\newcommand{\ra}{\ensuremath{\rightarrow}}
\newcommand{\lra}{\ensuremath{\longrightarrow}}
\newcommand{\FF}{\F}
\newcommand{\ff}{\mathfrak{f}}
\newcommand{\Z}{\mathbb{Z}}
\newcommand{\N}{\mathbb{N}}
\newcommand{\ch}{}
\newcommand{\R}{\mathbb{R}}
\renewcommand{\P}{\mathbb{P}}
\newcommand{\PP}{\mathbb{P}}
\newcommand{\pp}{\mathfrak{p}}
\newcommand{\C}{\mathbb{C}}
\newcommand{\Q}{\mathbb{Q}}
\newcommand{\ab}{\operatorname{ab}}
\newcommand{\Aut}{\operatorname{Aut}}
\newcommand{\gk}{\mathfrak{g}_K}
\newcommand{\gq}{\mathfrak{g}_{\Q}}
\newcommand{\OQ}{\overline{\Q}}
\newcommand{\Out}{\operatorname{Out}}
\newcommand{\End}{\operatorname{End}}
\newcommand{\Gal}{\operatorname{Gal}}
\newcommand{\CT}{(\mathcal{C},\mathcal{T})}
\newcommand{\lcm}{\operatorname{lcm}}
\newcommand{\Div}{\operatorname{Div}}
\newcommand{\OO}{\mathcal{O}}
\newcommand{\rank}{\operatorname{rank}}
\newcommand{\tors}{\operatorname{tors}}
\newcommand{\IM}{\operatorname{IM}}
\newcommand{\CM}{\mathbf{CM}}
\newcommand{\HS}{\mathbf{HS}}
\newcommand{\Frac}{\operatorname{Frac}}
\newcommand{\Pic}{\operatorname{Pic}}
\newcommand{\coker}{\operatorname{coker}}
\newcommand{\Cl}{\operatorname{Cl}}
\newcommand{\loc}{\operatorname{loc}}
\newcommand{\GL}{\operatorname{GL}}
\newcommand{\PGL}{\operatorname{PGL}}
\newcommand{\PSL}{\operatorname{PSL}}
\newcommand{\Frob}{\operatorname{Frob}}
\newcommand{\Hom}{\operatorname{Hom}}
\newcommand{\Coker}{\operatorname{\coker}}
\newcommand{\Ker}{\ker}
\newcommand{\g}{\mathfrak{g}}
\newcommand{\sep}{\operatorname{sep}}
\newcommand{\new}{\operatorname{new}}
\newcommand{\Ok}{\mathcal{O}_K}
\newcommand{\ord}{\operatorname{ord}}
\newcommand{\mm}{\mathfrak{m}}
\newcommand{\Ohell}{\OO_{\ell^{\infty}}}
\newcommand{\cc}{\mathfrak{c}}
\newcommand{\ann}{\operatorname{ann}}
\renewcommand{\tt}{\mathfrak{t}}
\renewcommand{\cc}{\mathfrak{a}}
\renewcommand{\aa}{\mathfrak{a}}
\newcommand\leg{\genfrac(){.4pt}{}}
\renewcommand{\gg}{\mathfrak{g}}
\renewcommand{\O}{\mathcal{O}}
\newcommand{\Spec}{\operatorname{Spec}}
\newcommand{\rr}{\mathfrak{r}}
\newcommand{\rad}{\operatorname{rad}}
\newcommand{\SL}{\operatorname{SL}}
\def\hh{\mathfrak{h}}
\newcommand{\zz}{\mathbf{z}}
\newcommand{\kk}{\underline{k}}
\newcommand{\uomega}{\underline{\omega}}
\newcommand{\AGL}{\operatorname{AGL}}

\begin{abstract}
We pursue various restricted variable generalizations of the Chevalley-Warning theorem for low degree polynomial systems over a finite field.  Our first such result involves variables restricted to Cartesian products of the Vandermonde subsets of $\F_q$ defined by G\'acs-Weiner and Sziklai-Tak\'ats.  We then define an invariant 
$\uomega(X)$ of a nonempty subset of $\F_q^n$.  Our second result involves $X$-restricted variables when the degrees of 
the polynomials are small compared to $\uomega(X)$.  We end by exploring various classes of subsets for which $\uomega(X)$ can be 
bounded from below.
\end{abstract}
\maketitle


\section{Introduction}
\noindent
We denote the set of  non-negative integers by $\N$ and the set of positive integers by $\Z^+$.  Throughout, $\F_q$ denotes a finite field of order $q$ and characteristic $p$.  For a finite, nonempty subset $X$ of a field $F$, we put 
\[ \varphi_X(t) \coloneqq \prod_{x \in X} (t-x), \]
so $\varphi_X$ generates the ideal of all polynomials $f \in F[t]$ that vanish identically on $X$.
\\ \\
Our point of departure is the following results of Chevalley and Warning, published in consecutive articles in the same journal in 1935 \cite{Chevalley35}, \cite{Warning35}.\footnote{Some further information about the history of these results can be found in \cite[\S 1]{CGS21}.}

\begin{thm}
\label{CW}
 Let $f_1,\ldots,f_r \in \F_q[t_1,\ldots,t_n]$ be polynomials of degrees $d_1,\ldots,d_r \in \Z^+$, and 
suppose that $d \coloneqq \sum_{j=1}^r d_j < n$.  Put 
\[ Z = Z(f_1,\ldots,f_r) \coloneqq \{ x = (x_1,\ldots,x_n) \in \F_q^n \mid f_1(x) = \ldots = f_r(x) = 0\} \]
be the solution set of the polynomial system.  
\begin{itemize}
\item[a)] (Chevalley) We \emph{cannot} have $\# Z = 1$.
\item[b)] (Chevalley-Warning) We have $p \mid \# Z$.
\item[c)] (Warning) If $Z$ is nonempty, then $\# Z \geq q^{n-d}$.
\end{itemize}
\end{thm}
\noindent
The second author has taken the perspective (e.g. in \cite{Clark14}) that Theorem \ref{CW}a) is a precursor of the following celebrated result.

\begin{thm}[Combinatorial Nullstellensatz II \cite{Alon99}]
\label{CN}
Let $F$ be a field, let $n \in \Z^+$, let $d_1,\ldots,d_n \in \N$ and let $f \in F[t_1,\ldots,t_n]$ be a polynomial.  We suppose:
\begin{itemize}
\item[(i)] We have $\deg(f) = d_1 + \ldots + d_n$, and 
\item[(ii)] The coefficient of $t_1^{d_1} \cdots t_n^{d_n}$ in $f$ is nonzero.
\end{itemize}
Then, for any subsets $X_1,\ldots,X_n$ of $F$ with $\# X_i = d_i + 1$ for $1 \leq i \leq n$, there exists an $x = (x_1,\ldots,x_n) \in 
X \coloneqq \prod_{i=1}^n X_i$ such that $f(x) \neq 0$.  
\end{thm}
\noindent
To see the connection, we introduce the \textbf{Chevalley polynomial} 
\[ \chi = \chi(f_1,\ldots,f_r) \coloneqq \prod_{j=1}^r (1-f_j^{q-1}). \]
For all $x \in \F_q^n$, we have 
\[\chi(x) = \begin{cases} 1 & \text{if } x \in Z(f_1,\ldots,f_r) \\ 0 & \text{otherwise} \end{cases}. \]
Because the hypotheses and the conclusion 
of Theorem \ref{CW}a) are stable under translation of variables, to prove Theorem \ref{CW}a), seeking a contradiction we may assume 
that $Z = \{0\}$, and then the polynomial
\[ P \coloneqq \chi(t_1,\ldots,t_n) - \prod_{i=1}^n (1-t_i^{q-1}) \]
evaluates to $0$ for all $x \in \F_q^n$.  But our hypothesis $\sum_{j=1}^r d_j < n$ implies that $\deg(\chi) < (q-1)n$, 
so the coefficient of $t_1^{q_1} \cdots t_n^{q-1}$ in $P$ is  $(-1)^{n+1} \neq 0$.  This contradicts Theorem \ref{CN} with $X_1 = \ldots = X_n = \F_q$.  
\\ \\
In turn Theorem \ref{CN} motivates us to consider \textbf{restricted variable generalizations} of Theorem \ref{CW}: instead of considering 
solutions to our polynomial system $f_1 = \ldots = f_r = 0$ on all of $\F_q^n$, we choose a subset $X \subset \F_q^n$ and look only 
at 
\[ Z_X \coloneqq \{ x = (x_1,\ldots,x_n) \in X \mid f_1(x) = \ldots = f_r(x) = 0\}. \]
A similar argument gives the following result of Schauz \cite{Schauz08} and Brink \cite{Brink11}.  

\begin{thm}
\label{RESCHEV}
Let $F$ be a field $\chi \in F[t_1,\ldots,t_n]$, let $X_1,\ldots,X_n$ be nonempty finite subsets of $F$, and put $X \coloneqq \prod_{i=1}^n X_i$.  
\begin{itemize}
\item[a)] If $\deg(\chi) < \sum_{i=1}^n 
(\# X_i-1)$, then $\# \{x \in \prod_{i=1}^n X_i \mid \chi(x) \neq 0\} \neq 1$.  
\item[b)] (Restricted Variable Chevalley Theorem) Let $f_1,\ldots,f_r \in \F_q[t_1,\ldots,t_n]$ have positive degrees.  If  
\begin{equation}
\label{RESCHEVEQ}
(q-1) \sum_{j=1}^r \deg(f_j) < \sum_{i=1}^n (\# X_i-1), 
\end{equation}
then we have $\# Z_X \neq 1$.
\end{itemize}
\end{thm}
\noindent
Just as parts b) and c) of Theorem \ref{CW} give two different generalizations of Theorem \ref{CW}a), it is natural to ask for restricted variable generalizations of Theorem \ref{CW}b) and of Theorem \ref{CW}c), each generalizing Theorem \ref{RESCHEV}b).  \\ \indent The latter has been attained:
in \cite[Thm. 1.6]{CFS17}, Clark-Forrow-Schmitt give a \textbf{Restricted Variable Warning Theorem}:\footnote{In fact \cite[Thm. 1.6]{CFS17} is a further ring theoretic generalization motivated by work of Brink \cite{Brink11}, but in the present paper we will only 
consider polynomials over a field.}
the conclusion of Theorem \ref{RESCHEV}b) is strengthened to: either $Z_X = \varnothing$ or $\# Z_X$ 
is at least a certain function of $n,d_1,\ldots,d_r,\# X_1,\ldots,\# X_n$ that is at least $2$ when (\ref{RESCHEVEQ}) holds. 
\\ \indent
The former case is addressed by the following result of \cite{Clark14}.

\begin{thm} \cite[Thm. 19]{Clark14} 
\label{RESCW0}
Let $f_1,\ldots,f_r \in \F_q[t]$ be polynomials of positive degrees.  For $1 \leq i \leq n$ let $X_i$ be a nonempty subset of $\F_q$.  
Put $X \coloneqq \prod_{i=1}^n X_i$ and 
\[ \varphi_i(t_i) \coloneqq \varphi_{X_i}(t_i) = \prod_{x_i \in X_i} (t_i-x_i). \]
Let 
\[ Z_X \coloneqq \{ x = (x_1,\ldots,x_n) \in X \mid P_1(x) = \ldots = P_r(x) = 0\}. \]
Suppose 
\begin{equation}
\label{RESCW0EQ1}
(q-1)\sum_{j=1}^r \deg(f_j) < \sum_{i=1}^n (\# X_i-1).
\end{equation}  
Then as elements of $\F_q$ we have 
\begin{equation}
\label{RESCW0EQ2}
 \sum_{x \in Z_X} \frac{1}{\prod_{i=1}^n \varphi_i'(x_i)} = 0. 
\end{equation}
\end{thm}
\noindent
Theorem \ref{RESCW0} is proved using the Coefficient Formula \cite[Thm. 3]{Clark14}, a refinement of Theorem \ref{CN} due to  Schauz \cite{Schauz08}, Lason
\cite{Lason10} and Karasev-Petrov \cite{Karasev-Petrov12}.
\\ \indent
In the case that $X_1 = \ldots = X_n = \F_q$, we have $\varphi_i(t_i) = t_i^q-t_i$, so $\varphi_i'(t_i) = -1$.  In this case (\ref{RESCW0EQ2}) 
becomes the assertion that $(-1)^n \# Z = 0$ in $\F_q$, so $p \mid \# Z$, and we recover Theorem \ref{CW}b).   Moreover, 
(\ref{RESCW0EQ2}) certainly implies that $\# Z_X \neq 1$, so we recover Theorem \ref{RESCHEV}b).  Thus Theorem \ref{RESCW0} 
is a simultaneous generalization of the Chevalley-Warning Theorem and the Restricted Variable Chevalley Theorem, so in \cite{Clark14} this result is called the ``Restricted Variable Chevalley-Warning Theorem.''
\\ \indent
In this note we wish to reopen the question of what should, or could, constitute a Restricted Variable Chevalley-Warning Theorem.  A distinguishing feature of Theorem \ref{CW}b) is that it gives a ``$p$-adic inequality on $\# Z$'' -- i.e., a $p$-divisibility on $\# Z$ -- a feature that does not seem to be present in Theorem \ref{RESCW0}.  Clearly the condition (\ref{RESCW0EQ1}) of Theorem \ref{RESCW0} is not in general sufficient to deduce $p \mid \# Z_X$: for instance fix $1 \leq I \leq n$, suppose $0 \in X_I$ and 
take $r = 1$ and $f_1 = t_I$: then $Z_X = \{ (x_1,\ldots,x_n) \in 
\prod_{i=1}^n X_i \mid x_I = 0\}$, so $\# Z_X = \prod_{i \neq I} \# X_i$.  This and similar examples show that it is not reasonable 
to expect $Z_X$ to be divisible by $p$ unless $p \mid \# X_i$ for all $i$.  Moreover, beyond any condition on the sizes of the $X_i$'s, we want somehow to take their structure into account.  
\\ \\
Here is an example of how to do this, a generalization of Theorem \ref{CW}b) due to Aichinger-Moosbauer \cite[p. 62]{Aichinger-Moosbauer21}.

\begin{prop}
\label{COSETPROP}
In the setting of Theorem \ref{RESCW0}, suppose moreover that each $X_i$ is a coset: that is, 
there is a subgroup $G_i \subseteq (\F_q,+)$ and $a_i \in \F_q$ such that $X_i = a_i + G_i$ for all $1 \leq i \leq n$.  If (\ref{RESCW0EQ1}) holds, then 
$p \mid \# Z_X$.
\end{prop}
\begin{proof} 
If for all $1 \leq i \leq n$ there is $c_i \in \F_q^{\times}$ such that for all $x_i \in X_i$ we have $\varphi_i'(x_i) = c_i$, then we have 
\[ \sum_{x \in Z_X} \frac{1}{\prod_{i=1}^n \varphi_i'(x_i)} = \sum_{x \in Z_X} \frac{1}{c_1 \cdots c_n} = \frac{ \# Z_X}{c_1 \cdots c_n}, \]
so when we apply Theorem \ref{RESCW0} the conclusion is $0 = \frac{\# Z_X}{c_1 \cdots c_n}$ and thus $p \mid \# Z_X$.  \\ \indent
For any finite nonempty subset $X$ of a field $F$, if $\varphi_X(t) \coloneqq \prod_{x \in X} (t-x)$, then for all $x \in X$ we have 
\[ \varphi_X'(x) = \prod_{y \in X \setminus \{x\}} (x-y). \]
Taking  $X = X_i = a_i + G_i$, for all $x_i \in a_i + G_i$ we get 
\[ \varphi_{X_i}'(x_i) = \prod_{y \in (a_i + G_i) \setminus \{x_i\}} (x_i-y) = \prod_{x \in G_i \setminus \{0\}} x \eqqcolon c_i; \]
that is to say, the value of $\varphi_{X_i}'(x_i)$ does not depend upon the choice of $x_i \in X_i$.  
\end{proof}
\noindent
The argument of Proposition \ref{COSETPROP} will work whenever $X = \prod_{i=1}^n X_i$ and for all $1 \leq i \leq n$ we 
have that $\varphi_{X_i}'(x) = \varphi_{X_i}'(y)$ for all $x,y \in X_i$.  Our first main result is that, for any field $F$ 
of characteristic $p > 0$ and any finite subset $X \subset F$ of cardinality at least $2$, we have $\varphi_X'(x) = \varphi_X'(y)$ for all $x,y \in X$ if and only if $X$ has size divisible by $p$ and is a \textbf{Vandermonde set}\footnote{In \S 2 we will give a self-contained treatment of Vandermonde sets.} in the sense of G\'acs-Weiner \cite{GW03} and Sziklai-Tak\'ats \cite{Sziklai-Takats08}.    
We deduce the following result.

\begin{thm}[Restricted Chevalley-Warning for Vandermonde Sets]
\label{VCW}
For $1 \leq i \leq n$, let $X_i \subset \F_q$ be a Vandermonde set of size divisible by $p$.   Let $f_1,\ldots,f_r \in \F_q[t_1,\ldots,t_n]$ be polynomials of positive degree such that
\[ (q-1)\sum_{j=1}^r \deg(f_j) < \sum_{i=1}^n (\# X_i-1). \]
Let 
\[ Z_X \coloneqq  \{ x \in \prod_{i=1}^n X_i \mid P_1(x) = \ldots = P_r(x) = 0\}. \]
Then $p \mid \# Z_X$.
\end{thm}
\noindent
We also give some examples of Vandermonde subsets of $\F_q$ of cardinality divisible by $p$ that are not cosets of additive subgroups.
\\ \\
Let $F$ be a field of characteristic $p > 0$.  In \S 3, for any $n \in \Z^+$ and any finite, nonempty subset $X \subset F^n$,
we define an invariant $\uomega(X) \in \N$ that is (almost) a multivariate generalization of the invariant $\omega(Y)$ of a subset 
$Y \subset \F_q$ considered by G\'acs-Weiner and Sziklai-Tak\'ats.  We show the following result:

\begin{thm}
\label{RESCW}
Let $X \subset \F_q^n$ be a nonempty subset, and let $f_1,\ldots,f_r \in \F_q[t_1,\ldots,t_n]$ be polynomials of positive degree.  If 
\begin{equation}
\label{RESCW1EQ}
(q-1) \sum_{j=1}^r \deg(f_j) < \uomega(X), 
\end{equation}
then 
\[ p \mid \# Z_X = \# \{ x \in X \mid f_1(x) = \ldots = f_r(x) = 0 \}. \]
\end{thm}
\noindent
Of course this focuses attention on what we know about $\uomega(X)$.  We have $\uomega(X) \geq 1$ iff $p \mid \# X$.  
Moreover, if $X_i$ is the 
projection of $X$ onto its $i$th factor, then we have 
\begin{equation}
\label{OPTIMALEQ1}
 \uomega(X) \leq \sum_{i=1}^n (\# X_i-1). 
\end{equation}
We call a subset $X \subset F^n$ \textbf{optimal} if equality occurs in (\ref{OPTIMALEQ1}).  In Lemma \ref{NEWRCW5} we show that if $X = \prod_{i=1}^n X_i$ and each $X_i$ is Vandermonde of size divisible by $p$, then $X$ is otpimal.  Thus Theorem \ref{RESCW} implies 
Theorem \ref{VCW}.  Because of this we suggest that Theorem \ref{RESCW} is more deserving of the name ``Restricted Variable Chevalley-Warning Theorem'' than Theorem \ref{RESCW0}.  

\begin{remark}
The main observation of this paper is that a Restricted Variable Chevalley Theorem for a subset $X \subset \F_q^n$ should take into 
account the ``structure'' of $X$ via an invariant that measures the vanishing of certain symmetric functions summed over $X$.   This observation is also the point of departure of a recent preprint of Nica \cite{Nica21}.  The main difference is that while we are interested 
in restricted variable theorems whose conclusion is $p \mid \# Z_X$, Nica is interested in improvements of the Combinatorial Nullstellensatz and the Coefficient Formula that take the structure of the ``grid'' $X = \prod_{i=1}^n X_i$ into account.
\\ \indent
His work and ours were done independently, and despite the similarity of content, have virtually no overlap.\footnote{In both \cite{Nica21} 
and the present work, Newton's identities are used to relate vanishing of power sums to vanishing of symmetric functions...as surely many others have done as well.}
\end{remark}

\section{Vandermonde sets}
\noindent
Vandermonde subsets of finite fields were defined by G\'acs-Weiner \cite{GW03} in characteristic $2$ and by Sziklai-Tak\'ats \cite{Sziklai-Takats08} in all characteristics.  In this section we study Vandermonde subsets over any field $F$ of characteristic 
$p > 0$.  

\subsection{Vandermonde Sets and their Polynomials}
Let $Y$ be a finite subset of $F$ of cardinality $r \geq 1$.   For $k \in \Z^+$ we put 
\[\pi_k(Y) \coloneqq \sum_{y \in Y} y^k. \]

\begin{lemma}
\label{SZIKLAI1}
Let $\{0\} \neq Y \subset F$ be a finite subset of size $r \geq 1$.
\begin{itemize}
\item[a)] If $p \mid r$ then $\pi_k(Y) \neq 0$ for some $1 \leq k \leq r-1$.  
\item[b)] In general we have $\pi_k(Y) \neq 0$ for some $1 \leq k \leq r$.
\end{itemize}
\end{lemma}
\begin{proof}
Write $Y = \{y_1,\ldots,y_r\}$.  \\
a) Consider the $r \times r$ Vandermonde matrix $V = V(y_1,\ldots,y_r)$ associated to $y_1,\ldots,y_r$, where the $i$-th row of $V$ is $(y_1^{i - 1}, \ldots, y_r^{i-1})$.  
Say $\pi_k(Y) = 0$ for 
all $1 \leq k \leq r-1$.
Then $V$ evaluated at the column vector $(1,\ldots,1)^T$ is $0$, since the the first entry of the 
product is $r = 0 \in F$ as $p \mid r$ and rest of the entries are $\pi_k(Y)$ for $1 \leq k \leq r-1$.  This contradicts the nonsingularity of $V$.  \\
b) Suppose first that $0 \notin Y$, and consider the matrix $\tilde{V}$ obtained from $V$ by multiplying the $j$th column by $y_j$, so 
$\det \tilde{V} = y_1 \cdots y_r \det(V) \neq 0$.   Again we must have that $\tilde{V} (1,\ldots,1)^T \neq 0$ which means that 
$\pi_k(Y) \neq 0$ for some $1 \leq k \leq r$.  If $0 \in Y$, let $Y^{\bullet} \coloneqq Y \setminus \{0\}$.   As we have just shown, 
there is $1 \leq k \leq r-1$ such that $\pi_k(Y^{\bullet}) \neq 0$, so then $\pi_k(Y) = \pi_k(Y^{\bullet}) \neq 0$.  
\end{proof}
\noindent
For the rest of this section $Y$ will denote a finite subset of $F$ of size $r \geq 2$.  \\ \indent
We denote by $\omega(Y)$ the least $k \in \Z^+$ such that $\pi_k(Y) \neq 0$.  Lemma \ref{SZIKLAI1} gives
\[\omega(Y) \leq \begin{cases}  r & \text{always} \\  r - 1 & \text{if } p \mid \# Y \end{cases}. \]  Following Sziklai-Tak\'ats \cite{Sziklai-Takats08} we say that $Y$ is \textbf{Vandermonde} if $\omega(Y) = r-1$ and is \textbf{super-Vandermonde} if $\omega(Y) = r$.  

\begin{lemma}
\label{SZIKLAI1.5}
Let $Y \subset F$ be finite of cardinality $r \geq 2$.  
\begin{itemize}
\item[a)] For all $k \in \Z^+$ and $\alpha \in F^{\times}$ we have $\pi_k(Y) = 0 \iff \pi_k(\alpha Y) = 0$.  Thus
$\omega(Y) = \omega(\alpha Y)$ and $Y$ is Vandermonde $\iff$ $\alpha Y$ is Vandermonde.  
\item[b)] Suppose $Y$ is Vandermonde.  Then for $\beta \in F^{\times}$ the translate $Y + \beta$ is Vandermonde iff $p \mid r$.  
\end{itemize}
\end{lemma}
\begin{proof}
a) Indeed $\pi_k(\alpha Y) = \alpha^k \pi_k(Y)$, and the rest is clear. \\
b) If $Y = \{y_1,\ldots,y_r\}$ then for all $1 \leq k \leq r-2$, expanding out each term in
\[ \pi_k(Y + \beta) = \sum_{i=1}^r (y_i + \beta)^k \]
gives a linear combination of $\pi_j(Y)$ for $j \leq k$ together with the final term $r \beta^k$, which is zero iff 
$p \mid r$.  
\end{proof}

\begin{prop}
\label{SZIKLAI2}
Let $F$ be a field of characteristic $p > 0$, let $Y = \{y_1,\ldots,y_r\} \subset F$ be finite of cardinality $r \geq 2$, and write 
\[\varphi_Y(t) = \prod_{i=1}^r (t-y_i) = t^r + \sum_{i=0}^{r-1} a_i t^i. \]
\begin{itemize}
\item[a)] The following are equivalent: 
\begin{itemize}
\item[(i)] The subset $Y$ is Vandermonde or super-Vandermonde. 
\item[(ii)] For $2 \leq i \leq r-1$, if $p \nmid i$ 
then $a_i = 0$. 
\end{itemize}
\item[b)]
The following are equivalent: 
\begin{itemize}
\item[(i)] The subset $Y$ is super-Vandermonde.  
\item[(ii)] For $1 \leq i \leq r-1$, if $p \nmid i$ then $a_i = 0$.  
\end{itemize}
\end{itemize}
\end{prop}
\begin{proof}
For independent indeterminates $x_1,\ldots,x_r$ and $0 \leq k \leq r$ let $s_k(x_1,\ldots,x_r)$ be the $k$th elementary symmetric function and let $\pi_k(x_1,\ldots,x_r) = \sum_{i=1}^r x_i^k$ be the $k$th power sum.  
Note that $a_i = (-1)^{r - i} s_{r - i}(x_1, \dots, x_r)$, for $0 \leq i \leq r - 1$. 
In the ring $\Z[x_1,\ldots,x_r]$ we have \textbf{Newton's identities} (see e.g. \cite{Zeilberger84} for a short proof): for all $1 \leq k \leq r$ we have
\[ k s_k(x_1,\ldots,x_r) = \sum_{i=1}^k (-1)^{i-1} s_{k-i}(x_1,\ldots,x_r) \pi_k(x_1,\ldots,x_r) \]
and 
\[ \pi_k(x_1,\ldots,x_r) = (-1)^{k-1} k s_k(x_1,\ldots,x_r) + \sum_{i=1}^{k-1} (-1)^{k-1+i} s_{k-i}(x_1,\ldots,x_r) \pi_i(x_1,\ldots,x_r). \]
From these we see that for elements $y_1,\ldots,y_r$ in a field of characteristic $p$, for any $1 \leq k \leq r$, the vanishing of the power sums 
$\pi_i(y_1,\ldots,y_r)$ for $1 \leq i \leq k$ is equivalent to the vanishing of the elementary symmetric functions $s_i(y_1,\ldots,y_r)$ for $1 \leq i \leq k$ and $p \nmid i$.  The result follows easily.
\end{proof}

\begin{thm}
\label{KEYVANDERLEMMA}
Let $F$ be a field of characteristic $p > 0$.  For a nonempty subset $Y \subset F$ or size $r \geq 2$, the following are equivalent: 
\begin{itemize}
\item[(i)] There is $g \in F[t]$ and $c \in F^{\times}$ such that $\varphi_Y = g(t^p) + ct$.  
\item[(ii)] The polynomial $\varphi_Y'$ is a nonzero constant (i.e., has degree zero).  
\item[(iii)] The polynomial $\varphi_Y'$ is constant on $Y$. 
\item[(iv)] The set $Y$ is Vandermonde of size a multiple of $p$.
\end{itemize}
\end{thm}
\begin{proof}
(i) $\iff$ (iv): It follows from Proposition \ref{SZIKLAI2}a) that $Y$ is Vandermonde of size divisible by $p$ if and only if $\varphi_Y = 
g(t^p) + ct$ for some $c \in F$.  Since $p \mid \# Y$, by Lemma \ref{SZIKLAI1} the set 
$Y$ is \emph{not} super-Vandermonde, so $c \neq 0$.   \\
(i) $\implies$ (ii): We have $(g(t^p)+ct)' = c$. \\
(ii) $\implies$ (iii): This is immediate.  \\
(iii) $\implies$ (i): Let $r = \# Y = \deg \varphi_Y$, and let $c \in F$ be the constant value of $\varphi_Y'$ on $Y$.  Since the polynomial 
$\varphi_Y(X)$ is separable, for all $y \in Y$ we have $\varphi_Y'(y) \neq 0$, so $c \neq 0$.  Consider $f \coloneqq \varphi_Y - ct$.  
Then $f'$ has degree at most $r-1$ and vanishes identically on $Y$, a set of size $r$, so $f'$ is the zero polynomial.  Thus there is $g \in F[t]$ such that $\varphi_Y - ct = f = g(t^p)$.  
\end{proof}
\noindent
As explained in the introduction, Theorem \ref{KEYVANDERLEMMA} implies Theorem \ref{VCW}.  In fact we get the following more 
general result that applies to any field of characteristic $p$.

\begin{thm}[Vandermonde Coefficient Formula]
Let $F$ be a field of prime characteristic $p$.  For $1 \leq i \leq n$ let $X_i \subset F$ be a finite nonempty Vandermonde 
subset of cardinality $d_i+1$ divisible by $p$. For $1 \leq i \leq n$, let $c_i = \varphi_i' \in F^{\times}$.  Put $d = (d_1,\ldots,d_n)$ and $X = \prod_{i=1}^n X_i$.  Let $f \in F[t_1,\ldots,t_n]$ be $d$-topped -- this means that the monomial $t_1^{d_1} \cdots t_n^{d_n}$ 
does not divide any other monomial appearing in $f$ with a nonzero coefficient, and this holds when $\deg(f) \leq d$ --
and let $c_d(f)$ be the coefficient of $t_1^{d_1} \cdots t_n^{d_n}$ in $f$.  Then we have 
\[ \left( c_1 \cdots c_n \right) c_d(f) = \sum_{x \in X} f(x). \]
\end{thm}
\begin{proof}
This follows from a suitable version of the Coefficient Formula: e.g. from \cite[Thm. 3.9]{Clark14}.
\end{proof}


\subsection{More on Cosets}
Suppose $Y \subset F$ is a coset of a finite subgroup of $(F,+)$.  Then $\varphi_Y'$ is constant on $Y$: this was shown in the proof of Proposition \ref{COSETPROP} when $F$ is finite, but the argument holds verbatim.  Applying Theorem \ref{KEYVANDERLEMMA}, we deduce:

\begin{cor}
\label{COSETVANDER}
Let $Y$ be a coset of a finite subgroup of $(F,+)$.  Then $Y$ is a Vandermonde set.
\end{cor}
\noindent
Corollary \ref{COSETVANDER} was shown for additive subgroups in $\F_{2^a}$ by G\'acs-Weiner \cite[Ex. 2.3(i)]{GW03} and for cosets of additive subgroups in $\F_q$ by Sziklai-Tak\'ats \cite[Prop. 1.8(i)]{Sziklai-Takats08}.  Our proof is different from both of theirs, which draw on the theory of 
additive polynomials.  With this approach one can deduce more about $\varphi_Y$ when $Y$ is a coset, as we now explain.
\\ \\
Suppose $F$ contains the finite field $\F_q$.  A polynomial $f \in F[t]$ is \textbf{functionally $\F_q$-linear} if the induced map $E(f): F \ra F$ given by $x \mapsto f(x)$ is an $\F_q$-vector space endomorphism.  In characteristic zero, if the induced map $E(f): F \ra F$ were
even a group homomorphism then $f$ would have to be of the form $f(t) = at$ for some $a \in F$; however in positive characteristic 
the functionally $\F_q$-linear polynomials are precisely those of the form $\sum_{i=0}^n a_i t^{q^i}$ for $a_i \in F$.  We say that a polynomial $f \in F[t]$ is \textbf{additive} if it is functionally $\F_p$-linear: equivalently, $E(f): (F,+) \ra (F,+)$ is a group homomorphism.    
A polynomial $f \in F[t]$ is \textbf{functionally $\F_q$-affine} if $f-f(0)$ is functionally $\F_q$-linear.  

\begin{thm}
Let $Y \subset F$ be a finite nonempty subset.  
\begin{itemize}
\item[a)] 
The following are equivalent:
\begin{itemize}
\item[(i)] The subset $Y$ is an $\F_q$-subspace of $F$.
\item[(ii)] The polynomial $\varphi_Y$ is functionally $\F_q$-linear.  
\end{itemize}
\item[b)] The following are equivalent:
\begin{itemize}
\item[(i)] The subset $Y$ is a coset of an $\F_q$-subspace of $F$.
\item[(ii)] The polynomial $\varphi_Y$ is functionally $\F_q$-affine. 
\end{itemize}
\end{itemize}
\end{thm}
\begin{proof}
When $F$ is finite, these results are special cases of \cite[Thms. 3.56 and 3.57]{LN97}.   The proofs given there do not use 
the finiteness of $F$.
\end{proof}
\noindent
In particular, if $Y$ is a coset of a finite subgroup of $(F,+)$ then $\varphi_Y-\varphi_Y(0)$ is an additive polynomial, so there 
are $a_0,\ldots,a_n,b \in F$ such that 
\[ \varphi_Y(t) = \sum_{i=0}^n a_i t^{p^i} + b. \]
Notice that this is stronger than $\varphi_Y$ just being of the form $g(t^p) + ct$.  If $F \supseteq \F_q$ and $Y$ is a coset of an $\F_q$-subspace, then there are $a_0,\ldots,a_n,b \in F$ such that 
\[ \varphi_Y(t) = \sum_{i=0}^n a_i t^{q^i} + b, \]
a stronger conclusion still.

\subsection{More on Vandermonde Sets}
Theorem \ref{VCW} is a generalization of Proposition \ref{COSETPROP}.  How much of an improvement is it?  This comes down to asking how many more Vandermonde subsets of $\F_q$ of cardinality divisible by $p$ there are than cosets of additive subgroups.  In \cite[Prop. 1.8(ii)]{Sziklai-Takats08}, Sziklai and Tak\'ats construct a family 
of such Vandermonde sets over $\F_{q^2}$ with $p = 2$ that are, in general, not cosets of additive subgroups. 
In particular, certain geometrical objects in finite projective planes known as hyperovals (see \cite{BCP06} and the references therein for the list of known infinite families of hyperovals), give rise to such Vandermonde sets.
There are further examples of Vandermonde sets known for small fields (see for example \cite[Example 7]{Abdukhalikov-Ho19}), but a full classification is out of reach (simply because a full classification of hyperovals appears to be out of reach \cite{Vandendriessche}).  

We ask the following question related to the enumeration of Vandermonde sets. 

\begin{ques}
\label{VANDERMONDEQUES}
For a prime number $p$ and $a \in \Z^+$, let $V(p,n)$ be the number of subsets $X \subset \F_{p^n}$ that are Vandermonde of size divisible by $p$, and let $C(p,n)$ be the number of subsets $X \subset \F_{p^n}$ that are cosets of subgroups of $(\F_{p^n},+)$.  Straightforward calculation gives 
\[ C(p,n) = \sum_{d=0}^n p^{(1-d)(n-d)} \frac{\# \GL_n(\F_p)}{\# \GL_d(\F_p) \# \GL_{n-d}(\F_p)}, \]
where
\[ \# \GL_n(\F_p) = \prod_{i=1}^n (p^n-p^{n-i+1}). \]
Is it true that for each fixed $p$ we have $C(p,n) = o(V(p,n))$ as $n \ra \infty$?
\end{ques}

\section{The Invariant $\uomega(X)$}
\noindent
Again let $F$ be a field of characteristic $p > 0$.  For a function $f: F^n \ra F$ and a finite subset $X \subset F^n$ we put 
\[ \int_X f \coloneqq \sum_{x \in X} f(x) \in F. \]
A polynomial $P \in F[t_1,\ldots,t_n]$ determines a function $E(P): x \in F^n \mapsto P(x) \in F$.  We put $\int_X P \coloneqq \int_X E(P)$.   
\\ \\
For $\kk = (k_1,\ldots,k_n) \in \N^n$, we put $|\kk| \coloneqq k_1 + \ldots + k_n$.  For a finite nonempty subset $X \subseteq F^n$ and $\kk = (k_1,\ldots,k_n) \in \N^n$, let
\[ \pi_{\kk}(X) \coloneqq \int_X t_1^{k_1} \cdots t_n^{k_n} = \sum_{x = (x_1,\ldots,x_n) \in X} x_1^{k_1} \cdots x_n^{k_n}. \]
Thus 
\begin{equation}
\label{PI0EQ}
 \pi_{\underline{0}}(X) = \int_X 1 = \# X \in F. 
\end{equation}
We put 
\[ \uomega(X) \coloneqq \inf \{ |\kk| \ \big{|} \ \pi_{\kk}(X) \neq 0\}. \]
Thus for $d \in \N$ we have $\uomega(X) \geq d+1$ iff $\int_X P = 0$ for every $P \in F[t_1,\ldots,t_n]$ of degree at most $d$.
In particular, by (\ref{PI0EQ}) we have $\uomega(X) \geq 1$ iff $\pi_{\underline{0}}(X) = 0$ iff $p \mid \# X$.  
\\ \\
Our definition allows $\uomega(X)$ to be infinite, which would occur if and only if $\pi_{\kk}(X) = 0$ for all $\kk \in \N^n$.  But the following result shows that this is never the case.

\begin{lemma}
\label{NEWRCW3}
Let $X \subset F^n$ be finite and nonempty.  For $1 \leq i \leq n$, let
\[ X_i \coloneqq \{ a \in F \mid \exists \  (x_1,\ldots,x_{i-1},a_,x_{i+1},\ldots,x_n) \in X \} \]
be the projection of $X$ onto its $i$th coordinate.  Then we have 
\begin{equation}
\label{VECTORVANEQ}
 \uomega(X) \leq \sum_{i=1}^n \left( \#X_i -1 \right). 
\end{equation}
\end{lemma}
\begin{proof}
Choose $x = (x_1,\ldots,x_n) \in X$, and put 
\[\delta_{X,x} \coloneqq \prod_{i=1}^n \prod_{y_i \in X_i \setminus \{x_i\}}
\frac{t_i-y_i}{x_i-y_i}. \] 
Then $\deg \delta_{X,x} = \sum_{i=1}^n \left( \# X_i - 1 \right)$.  For $y \in X$ we have $\delta_{X,x}(y) = \begin{cases} 1 & y = x \\ 0 & y \neq x \end{cases}$, so and $\int_X \delta_{X,x} = 1$.  
\end{proof}
\noindent
To be sure: when $n = 1$ and $Y \subseteq F$ is finite of size $r \geq 1$, we have $\uomega(Y) = \omega(Y)$ if $p \mid \# Y$, 
while if $p \nmid \# Y$ we have $\uomega(Y) = 0$ and $\omega(Y) \geq 1$.  (In fact, we did not define $\omega(Y)$ for subset of size 
$1$: extending the definition we gave, we would have $\omega(\{ x\}) = 1$ if $x \neq 0$, while $\omega( \{0\})$ would be infinite.)


\subsection{Proof of Theorem \ref{RESCW}}
 Once again, let $\chi \coloneqq \prod_{i=1}^r (1-f_i^{q-1})$ be Chevalley's polynomial, so for $x \in X$ we have \[\chi(x) = \begin{cases} 1 & x \in Z_X \\
0 & x \notin Z_X \end{cases}. \]  Moreover $\deg \chi = (q-1) \sum_{j=1}^r d_j < \uomega(X)$, so in $\F_q$ we have
\[ \# Z_X = \int_X \chi(x)  = 0 \in \F_q.\]
Since $\F_q$ has characteristic $p$, this yields $p \mid \# Z_X$.

\subsection{Optimal Subsets}
A finite nonempty subset $X \subseteq F^n$ is \textbf{optimal} if equality holds in (\ref{VECTORVANEQ}):
\[\uomega(X) = \sum_{i=1}^n \left( \# X_i - 1 \right). \] 
Thus when $n = 1$ we get that an optimal subset of $F$ is precisely a set that either has size $1$ or is a Vandermonde set of 
size divisible by $p$.

\begin{remark}
That $\F_q^n$ is optimal is the crux of Ax's ``Quick Proof of the Chevalley-Warning Theorem'' \cite[\S 2]{Ax64}, \cite[Thm. 1.1]{CGS21}.  
In some sense the proof of Theorem \ref{RESCW} corresponds to \emph{the rest} of Ax's proof...which is why it is so short.
\end{remark}

\begin{lemma}
\label{NEWRCW5}
For $1 \leq i \leq n$, let $X_i \subseteq F$ be finite nonempty, and put 
$X \coloneqq \prod_{i=1}^n X_i$.  \begin{itemize}
\item[a)] For all $\kk \in \N^n$ we have $\pi_{\kk}(X) = \prod_{i=1}^n \pi_{k_i}(X_i)$. 
\item[b)] Let $J \coloneqq \{1 \leq i \leq n \mid \# X_i \neq 1\}$.  Then we have 
\[\uomega(X) = \sum_{j \in J} \omega(X_j). \]
\item[c)] If each $X_i \subseteq \F_q$ is Vandermonde of size divisible by $p$, then \[\uomega(X) = \sum_{i=1}^n \left( \# X_i -1 \right), \]
so $X$ is optimal.  
\end{itemize}
\end{lemma}
\begin{proof}
a) We have 
\[ \pi_{\kk}(X) = \sum_{(x_1,\ldots,x_n) \in \prod_{i=1}^n X_i} x_1^{k_1} \cdots x_n^{k_n} = \prod_{i=1}^n \sum_{x_i \in X_i} x_i^{k_i} = \prod_{i=1}^n \pi_{k_i}(X_i). \]
Part b) follows in the case where $\# X_i \geq 2$ for all $i$, as does part c).  
If $j \in \{1,\ldots,n\} \setminus J$ and $\kk \in \N^n$ is such that $\pi_{\kk}(X) \neq 0$, then $\pi_{(k_1,\ldots,k_{j-1},0,k_{j+1},\ldots,k_n)}(X) \neq 0$.  If $\kk \in \N^n$ is such that $\kk_j = 0$ for all $j \notin J$, then let $\hat{\kk} \in \N^J$ be the corresponding tuple with the indices outside of $J$ removed.  We have 
\[ \pi_{\kk}(X) = \pi_{\hat{\kk}}(\prod_{j \in J} X_j) = \sum_{j \in J} \omega(X_j). \qedhere \]
\end{proof}
\noindent
Lemma \ref{NEWRCW5} and Theorem \ref{RESCW} together imply Theorem \ref{VCW}.  

\subsection{The invariant $\uomega(X)$ of a coset $X$}
Now we look more closely at $\uomega(X)$ for a coset $X$ of a finite additive subgroup of $F$.  The following result reduces us to studying finite subgroups $G \subset (F^n,+)$.

\begin{lemma}[Affine Invariance]
\label{NEWRCW4}
Let $\operatorname{AGL}_n(F) = F^n \rtimes \GL_n(F)$ be the group of affine transformations of $F^n$.  For all $\sigma \in \operatorname{AGL}_n(F)$ and all finite nonempty $X \subset F^n$ we have $\uomega(\sigma(X)) = \uomega(X)$.
\end{lemma}
\begin{proof}
The group $\operatorname{AGL}_n(F)$ acts on polynomials: if \[\sigma \in \operatorname{AGL}_n(F): (x_1,\ldots,x_n) \in \F_q^n \mapsto  (a_1 + \sum_{j=1}^n m_{1,j} x_j,\ldots,a_n + 
\sum_{j=1}^n m_{n,j} x_j)\] and $P \in F[t_1,\ldots,t_n]$ then 
\[ P_{\sigma} = P(a_1 + \sum_{j=1}^n m_{1,i} t_i,\ldots,a_n + 
\sum_{j=1}^n m_{n,j} t_i). \]
This action preserves the degree.  So if $\deg(P) < \uomega(X)$ then 
\[ \int_{\sigma(X)} P = \int_X P_{\sigma} = 0, \]
so $\uomega(\sigma(X)) \leq \uomega(X)$.  Applying this with $\sigma^{-1}$ and $\sigma(X)$ in place of $\sigma$ and $X$ we deduce that $\uomega(\sigma(X)) = \uomega(X)$.
\end{proof}

\begin{example}
\label{DIAGEX}
Let $X \subset \F_p^2$ be a nontrivial, proper coset, so $\# X = p$.  The group $\operatorname{AGL}_2(\F_p)$ acts transitively on all nontrivial proper cosets of $\F_p^2$, so by Lemma \ref{NEWRCW4} all such cosets have the same invariant $\uomega(X)$; taking 
$X = \F_p \times \{0\}$ and applying Lemma \ref{NEWRCW5} we find that this common value is $p-1$.  However, for the subgroup
\[\Delta \coloneqq \{ (a,a) \mid a \in \F_p \}  \]
we have $\uomega(\Delta) = (p-1) < 2(p-1) = (\# \F_p - 1) + (\# \F_p-1)$.  Thus optimality is \emph{not} $\AGL_n(F)$-invariant.  More precisely optimality is translation-invariant and is not generally $\GL_n(F)$-invariant.
\end{example}

\begin{prop}
\label{EASYLINEARPROP}
Let $G \subset \F_q^n$ be an additive subgroup.  
\begin{itemize}
\item[a)] If $G$ is an $\F_q$-subspace, then $\uomega(G) = (\dim_{\F_q}(G))(q-1)$.
\item[b)] If $q = p$, then $\uomega(G) = (\log_p(\# G))(p-1)$.
\end{itemize}
\end{prop}
\begin{proof}
a) Let $e \coloneqq \dim_{\F_q}(G)$.  Since $\GL_n(\F_q)$ acts transitively on $\F_q$-subspaces of dimension $e$, we have $\uomega(G) = \uomega(\F_q^e \times \{0\}^{n-e}) = e(q-1)$. \\
b) Every finite subgroup of a field of characteristic $p$ is an $\F_p$-subspace, of dimension $\log_p(\# G)$.  So part a) applies.
\end{proof}
\noindent
Thus we can compute $\omega(G)$ when $n = 1$ or $F = \F_p$.  The remaining case seems more interesting.  

\begin{example}
\label{RICHEXAMPLE}
There are subgroups $A_1,B_1,B_2$ of $(F,+)$ with 
\[ A_1 \cong (\Z/p\Z)^2,  \ B_1 \cong B_2 \cong \Z/p\Z. \]
We put 
\[ A \coloneqq A_1 \times \{0\}^{n-1}, \ B = B_1 \times B_2 \times \{0\}^{n-2}. \]
Then
\[ \uomega(A) = \uomega(A_1) = p^2-1 > 2(p-1) = \uomega(B_1 \times B_2) = \uomega(B), \]
even though $A \cong (\Z/p\Z)^2 \cong B$ and thus there is an $\F_p$-linear automorphism $\sigma: F \ra F$ such that 
$B = \sigma(A)$.  
\end{example}
\noindent
In light of Examples \ref{DIAGEX} and \ref{RICHEXAMPLE} it is not clear to us what form the determination of $\uomega(G)$ 
for an arbitrary finite subgroup $G \subset (F,+)$ should take.   The following is the best possible lower bound that takes only 
the isomorphism class of $G$ into account:

\begin{prop}[Aichinger-Moosbauer]
\label{AMGROUPPROP}
Let $F$ be a field of characteristic $p > 0$, and let $G \subset (F^n,+)$ be a finite subgroup of order $p^e$.  Then
\begin{equation}
\label{AMSUBGROUPEQ}
 \uomega(G) \geq e(p-1). 
\end{equation}
\end{prop}
\begin{proof}
This follows from \cite[Lemma 8.2 and Lemma 12.1]{Aichinger-Moosbauer21}.   
\end{proof}

\begin{remark}
Combining Proposition \ref{AMGROUPPROP} and Theorem \ref{RESCW}, we find that if $X \subset \F_q^n$ is a coset of a subgroup of 
order $p^e$ and $f_1,\ldots,f_r \in \F_q[t_1,\ldots,t_n]$ have positive degree, then if 
\[ (q-1)\left(\sum_{j=1}^r \deg(f_j)\right) < d(p-1), \]
then $p \mid \# Z_X$.  When $q = p$, this result implies Proposition \ref{COSETPROP}.  If $q = p^a$ with $a > 1$, then a better 
bound follows from \cite[Thm. 12.2]{Aichinger-Moosbauer21}: we have $p \mid \# Z_X$ if 
\begin{equation}
\label{BIGAMEQ}
 a \left( \sum_{j=1}^r \deg(f_j) \right) < p-1. 
\end{equation}
The low degree condition (\ref{BIGAMEQ}) has a different form than the low degree condition (\ref{RESCW0EQ1}) of Proposition 
\ref{COSETPROP}, but using an elementary convexity argument, Aichinger and Moosbauer show that if (\ref{RESCW0EQ1}) holds then so does (\ref{BIGAMEQ}).
\end{remark}

\subsection{Graphs of Functions}
We end by giving a class of subsets $X$ of $\F_q^n$ for which $\uomega(X)$ can be bounded from below.

\begin{prop}
\label{OMEGAGRAPHPROP}
\label{NEWRCW6}
Let $n \geq 2$.  For $f \in \F_q[t_1,\ldots,t_{n-1}]$ of degree $d \geq 1$, let 
\[ X_f \coloneqq \{ (x,f(x)) \in \F_q^n\} \]
be the graph of the associated function $E(f): \F_q^{n-1} \ra \F_q$.  Then:
\begin{itemize}
\item[a)] We have $\omega(X_f) \geq \frac{(n-1)(q-1)}{d}$.
\item[b)] If $\omega(X_f) = \frac{(n-1)(q-1)}{d}$ then $d \mid (n-1)(q-1)$.
\item[c)] If $n-1 \mid d \mid (n-1)(q-1)$ and the only monomial of degree $d$ in the support of 
$f$ is $t_1^{\frac{d}{n-1}} \cdots t_{n-1}^{\frac{d}{n-1}}$, then $\omega(X_f) = \frac{(n-1)(q-1)}{d}$.  
\item[d)] If $n = 2$, then $\omega(X_f) = \frac{q-1}{d}$ iff $d \mid q-1$.  
\end{itemize}
\end{prop}
\begin{proof}
a) For $\kk = (k_1,\ldots,k_n) \in \N$, we have $\pi_{\kk}(X_f) =$ 
\[\sum_{(x_1,\ldots,x_{n-1}) \in \F_q^{n-1} } x_1^{k_1} \cdots x_{n-1}^{k_{n-1}} f(x_1,\ldots,x_{n-1})^{k_n}= \int_{\F_q^{n-1}}
t_1^{k_1} \cdots t_{n-1}^{k_{n-1}} f(t_1,\ldots,t_{n-1})^{k_n}. \]
Since $\mu(\F_q^{n-1}) = (n-1)(q-1)$, if $\pi_{\kk}(X) \neq 0$ then \[k_1 + \ldots + k_{n-1} + d k_n = \deg(t^{k_1} \cdots t^{k_{n-1}} 
f^{k_n}) \geq (n-1)(q-1), \]
so 
\[ |\kk| = k_1 + \ldots + k_n \geq \frac{k_1}{d} + \ldots + \frac{k_{n-1}}{d} + k_n \geq \frac{(n-1)(q-1)}{d}. \] 
b) Clearly equality can only hold if $d \mid (n-1)(q-1)$.   \\
c) If the hypotheses hold, then the only monomial of degree $(n-1)(q-1)$ in the support of $f^{\frac{(n-1)(q-1)}{d}}$ is 
$t_1^{q-1}\cdots t_{n-1}^{q-1}$, so \[\pi_{(0,\ldots,0,\frac{(n-1)(q-1)}{d})} X_f = \int_{\F_q^{n-1}} f^{\frac{(n-1)(q-1)}{d}} \neq 0. \]
d) When $n = 2$ and $d \mid q-1$, the conditions of part c) hold.
\end{proof}
\noindent
When $d = 1$, the subset $X_f \subset \F_q^n$ is an affine $\F_q$-hyperplane, so it follows from Proposition \ref{EASYLINEARPROP}a) that $\omega(X_f) = 
(n-1)(q-1)$, so we get another case in which the bound of Proposition \ref{OMEGAGRAPHPROP}a) is sharp.  


\begin{thebibliography}{CCRS14}


\bibitem[AH19]{Abdukhalikov-Ho19} K. Abdukhalikov and D. Ho, \emph{Vandermonde sets and hyperovals}.  
\url{https://arxiv.org/pdf/1911.10798.pdf}

\bibitem[Al99]{Alon99} N. Alon, \emph{Combinatorial Nullstellensatz}.
Recent trends in combinatorics (M\'atrah\'aza, 1995).
Combin. Probab. Comput. 8 (1999), 7--29.


\bibitem[AM21]{Aichinger-Moosbauer21} E. Aichinger and J. Moosbauer, \emph{Chevalley-Warning type results on abelian groups.}
J. Algebra 569 (2021), 30--66. 




\bibitem[Ax64]{Ax64} J. Ax, \emph{Zeroes of polynomials over finite fields}.
Amer. J. Math. 86 (1964), 255--261.

\bibitem[BCP06]{BCP06} Bayens, L., Cherowitzo, W., Penttila, T. \emph{Groups of hyperovals in Desarguesian planes}.
 Innov. Incidence Geom. 6/7 (2007/08), 37--51.



\bibitem[Br11]{Brink11} D. Brink, \emph{Chevalley's theorem with restricted variables}. Combinatorica 31 (2011), 127--130.

\bibitem[CFS17]{CFS17} P.L. Clark, A. Forrow and J.R. Schmitt,
\emph{Warning's Second Theorem With Restricted Variables}. Combinatorica 37 (2017), 397--417. 

\bibitem[CGS21]{CGS21} P.L. Clark T. Genao and F. Saia, \emph{Chevalley-Warning at the Boundary}.  To appear, \emph{Expositiones Math.}

\bibitem[Ch35]{Chevalley35} C. Chevalley, \emph{D\'emonstration d'une hypoth\`ese de M. Artin.} Abh. Math. Sem. Univ. Hamburg 11 (1935), 73--75.



\bibitem[Cl14]{Clark14} P.L. Clark, \emph{The Combinatorial Nullstellens\" atze revisited.}
Electron. J. Combin. 21 (2014), no. 4, Paper 4.15, 17 pp. 




\bibitem[GW03]{GW03} A. G\'acs and Z. Weiner, \emph{On $(q+t,t)$-arcs of type $(0,2,t)$}.
Proceedings of the Conference on Finite Geometries (Oberwolfach, 2001).
Des. Codes Cryptogr. 29 (2003),  131--139.





\bibitem[KP12]{Karasev-Petrov12} R.N. Karasev and F.V. Petrov, \emph{Partitions of nonzero elements of a finite field into pairs}. Israel J. Math. 192 (2012), 143--156.


\bibitem[La10]{Lason10} M. Laso\'n, \emph{A generalization of combinatorial Nullstellensatz}. Electron. J. Combin. 17 (2010), no. 1, Note 32, 6 pp.

\bibitem[LN97]{LN97} R. Lidl and H. Niederreiter, \emph{Finite fields.} With a foreword by P. M. Cohn. Second edition. Encyclopedia of Mathematics and its Applications, 20. Cambridge University Press, Cambridge, 1997. 





\bibitem[Ni21]{Nica21} B. Nica, \emph{Polynomials over structured grids.} \url{https://arxiv.org/abs/2110.05616}






\bibitem[Sc08]{Schauz08} U. Schauz, \emph{Algebraically solvable problems: describing polynomials as equivalent to explicit solutions}.
Electron. J. Combin. 15 (2008), no. 1, Research Paper 10, 35 pp.


\bibitem[ST08]{Sziklai-Takats08} P. Sziklai and M. Tak\'ats, \emph{Vandermonde sets and super-Vandermonde sets.}
Finite Fields Appl. 14 (2008), 1056--1067.



\bibitem[Va19]{Vandendriessche} P. Vandendriessche, \emph{Classification of the hyperovals in $PG(2,64)$}. Electron. J. Combin. 26 (2019), no. 2, Paper No. 2.35, 12 pp.

\bibitem[Wa35]{Warning35} E. Warning, \emph{Bemerkung zur vorstehenden Arbeit von Herrn Chevalley}.  Abh. Math. Sem. Hamburg 11 (1935), 76--83.



\bibitem[Ze84]{Zeilberger84} D. Zeilberger, \emph{A combinatorial proof of Newton's identities.}
Discrete Math. 49 (1984), 319. 



\end{thebibliography}
\end{document}